\documentclass[11pt,letterpaper]{article}
\usepackage{settings}

\title{The structure of Gram matrices of sum of squares polynomials with restricted harmonic support}
\author{Mitchell Tong Harris\\ \href{mitchh@mit.edu}{mitchh@mit.edu}  
}

\begin{document}

\maketitle

\makeatletter
\def\blfootnote{\gdef\@thefnmark{}\@footnotetext}
\makeatother
\blfootnote{The author is with the Laboratory for Information and Decision Systems (LIDS) and Mathematics Department, Massachusetts Institute of Technology, Cambridge MA 02139. The author was supported by the National Science Foundation Graduate Research Fellowship under Grant No. 2141064.}

\begin{abstract}
Some sum of squares (SOS) polynomials admit decomposition certificates, or positive semidefinite Gram matrices, with additional structure. In this work, we use the structure of Gram matrices to relate the representation theory of $SL(2)$ to $SO(2)$. Informally, we prove that if $p$ is a sum of squares and lives in some of the invariant subspaces of $SO(2)$, then it has a positive semidefinite Gram matrix that lives in certain invariant subspaces of $SL(2)$. The tools used in the proof construction are of independent interest.
\end{abstract}

\section{Introduction}

Spaces of homogeneous polynomials are commonly used as representing spaces for matrix groups, yet little attention is paid to the additional structure the polynomials admit as \emph{functions}. In this work, we leverage a characteristic particular to functions -- that of global nonnegativity -- to provide a new correspondence between representations of $SL(2)$ and $SO(2)$. 

Representation theory is a rich subject with many excellent books of varying scopes. Some span theory that applies to large classes of groups \cite{etingof2011introduction, fulton2013representation}. The theory of Lie groups is sufficiently mature that still more books are devoted entirely to these structures \cite{humphreys2012introduction, knapp1996lie}. Finally, some groups are so nuanced that their representation theory is treated on its own \cite{lang2012sl2}. When this level of granularity is reached, often the story ends there.

Our study on the structure of sums of squares (SOS) of polynomials provides a new link between the representation theory of two particular groups. The special structure that polynomial spaces admit is that some polynomials may be nonnegative or SOS. The study of the proxy of SOS for nonnegativity goes back to Hilbert \cite{hilbert1888ueber}, and more recently has become more popular due to advances in algorithmically checking this condition that make it tractable \cite{parrilo2000structured,lasserre2001global}.

There are many connections between the study of symmetry with the study of SOS polynomials. Most fundamentally, the property of being SOS is coordinate-free -- applying any change of coordinates to the original polynomial results in an equivalent transformation of the coordinates of the SOS decomposition. This property was leveraged in \cite{lawrence2023learning}. A more sophisticated exploitation of symmetry of SOS polynomials to accelerate computation was carried out in \cite{gatermann2004symmetry}. Beyond leveraging symmetry for practical improvements, the structure of sum of squares decomposition has been relatively underexplored. One such advance provided in \cite{fawzi2016sparse} relates the Fourier support of the original polynomial to the Fourier support of those in the SOS decomposition, where both Fourier decompositions are with respect to the same group. In this work, we explore the structure of an SOS decomposition as it relates to a group different from the one that is used to characterize the original polynomial.

We prove a correspondence between representations of $SL(2)$ and $SO(2)$ that is roughly of the following form.
\begin{quote}If $p$ is a sum of squares and lives in some of the invariant subspaces of $SO(2)$, then it has a positive semidefinite Gram matrix that lives in certain invariant subspaces of $SL(2)$. 
\end{quote}
The precise version of this statement and its proof are given in Section~\ref{sec:gram-matrix-structure}. We build up to this result by providing the necessary background in Section~\ref{sec:preliminaries}. Finally, the two major proof tools, which are not original but deserve expounding on because of their versatility, are presented in Section~\ref{sec:tools}. 

\section{Preliminaries}\label{sec:preliminaries}

In this section we recall some background about polynomials, groups, and their representation theory. 
\subsection{Polynomials} 
We establish a connection between two Lie groups via polynomials. On the one hand, polynomial spaces are generic vector spaces of appropriate dimension, but on the other their properties as functions afford a special richness. We let $F(n,d)$ be the set of degree $d$ homogeneous polynomials in $n$ variables with real coefficients, commonly referred to as $n$-ary $d$-forms. 

\subsubsection{Matrix group actions} Throughout this paper, all groups act on polynomials the same way. Let $G$ be a matrix group. For any $A \in G$ and $p \in F(n,d)$, the group action is given by $A(p(x)) = p(A^{-1}x)$.

The group action is a linear map on the polynomial coefficients, which we want to write as an explicit matrix. To do this we make use of the following lifting operation. Let $N = {n + d -1 \choose d}$. The map $\mathbb{R}^n \to \mathbb{R}^N$ denoted $x^{[d]}$ sends $x$ to the vector whose entries are the monomial basis of $F(n,d)$. For example, \begin{equation*}
    \pmat{x_1 \\x_2}^{[1]} = \pmat{x_1 \\x_2}, \qquad \pmat{x_1 \\x_2}^{[2]} = \pmat{x_1^2 \\ x_1x_2 \\ x_2^2}, \qquad  \pmat{x_1 \\x_2}^{[3]} = \pmat{x_1^3 \\ x_1^2x_2 \\ x_1 x_2^2\\ x_2^3}.
\end{equation*}
Often it is convenient to instead work in the \emph{scaled monomial basis}, in which $x^\alpha$ is scaled by $\sqrt{\frac{(\sum_i \alpha_i)!}{\prod_i (\alpha_i!)}}$. For example $\pmat{x_1 & x_2}^{[d]^T} = \pmat{x_1^2 & \sqrt{2} x_1 x_2 & x_2^2}^T$. We will switch between the two bases depending on which is more convenient; if not stated the bases is either irrelevant or clear from context. In either case, a simple diagonal transformation allows us to convert between them.

Furthermore, if $A$ is $n \times n$, then the matrix $A^{[d]}$ is defined as the unique matrix such that 
\begin{equation*}
    A^{[d]}x^{[d]} = (Ax)^{[d]}
\end{equation*}
for any $x \in \mathbb{R}^n$. There exists an explicit description of the map in terms of the entries of $A$. See \cite{marcus1992survey, marcus1973finite, parrilo2008approximation} for further details. Finally we see that if $v$ is the vector of coefficients of the polynomial $p$, then $v^T x^{[d]} = p(x).$ Therefore, \begin{equation}\label{eqn:explicitaction} p(Ax) = v^T (Ax)^{[d]} = v^T A^{[d]^T} x^{[d]} = (A^{[d]}v)^T x^{[d]}, \end{equation} and so the transformed coefficients are $A^{[d]} v$.

\subsubsection{Sums of squares}
A polynomial $p \in F(n,2d)$ is a \emph{sum of squares} if it can be written as
\begin{equation*}
p = \sum_{j=1}^J q_j^2,
\end{equation*}
where each $q_j \in F(n, d)$. As exploited by \cite{parrilo2000structured, lasserre2001global}, a polynomial $p \in F(n,2d)$ is a sum of squares if and only if there exists an $N \times N$ positive semidefinite (psd) matrix $Q$ such that
\begin{equation} \label{eqn:sosgrammatrix}
p(x) = x^{[d]^T} Q x^{[d]}.
\end{equation}
In this case, $Q$ is a sum of squares decomposition certificate. We call any matrix satisfying \eqref{eqn:sosgrammatrix} a \emph{Gram} matrix for $p$ and sometimes denote such a matrix as $M(p)$. 

It is natural to ask how a Gram matrix transforms when the polynomial undergoes a matrix group action. Given a Gram matrix for $p(x)$ we can directly compute a Gram matrix for $p(Ax)$. Substituting in \eqref{eqn:sosgrammatrix}, we get 
\begin{equation} \label{eqn:gram-matrix-equivariance}
    p(Ax) = (Ax)^{[d]^T}Q (Ax)^{[d]} = x^{[d]^T}(A^{[d]^T}Q A^{[d]})x^{[d]},
\end{equation} which proves $A^{[d]^T}Q A^{[d]}$ is a Gram matrix.
Furthermore, congruence transformations preserve positive semidefiniteness, so if $Q$ is psd, then so is $A^{[d]^T}Q A^{[d]}$.

\subsection{Representation theory}\label{sec:background-rep-theory}
We recall some basic Harmonic analysis and representation theory of $SL(2, \mathbb{R})$. Throughout the paper we often drop the $\mathbb{R}$ because we exclusively work with the real versions of these groups.

\subsubsection{Harmonic analysis}
Let $H_{n,k} = \{p \in F_{n,d} \, : \, \nabla p = 0\}$, where $\nabla = \sum_{i=1}^n \frac{\partial^2}{\partial x_i^2}$. Let $\Vert x \Vert^2 = \sum_{i=1}^n x_i^2$. Then any polynomial $p \in F_{n, 2d}$ has a unique \emph{Harmonic} decomposition as 
\begin{equation}
p = p_0 \Vert x \Vert^{2d} + p_2 \Vert x \Vert^{2(d-1)} + \cdots + p_{2d},
\end{equation}
where $p_k \in H_{n,k}$. This famous result is Theorem 2.1 in \cite{stein1971introduction} and Theorem 5.7 in \cite{axler2013harmonic}.

Each of the spaces $H_{n,k}$ is special because of their relationship with $SO(n, \mathbb{R})$; they are irreducible invariant subspaces of $F(n,d)$ under the action of $SO(n, \mathbb{R})$. This means that if $p \in H_{n,k}$ then $gp \in H_{n,k}$ for any $g \in SO(n, \mathbb{R})$. Furthermore, there is no smaller nonempty subset of $H_{n,k}$ with this property. 

\subsubsection{Tensor products}\label{sec:prelimtensorproducts} The basic idea of tensor products enters into our picture via the tensor products of a representation. A representation of a group $G$ is a vector space $V$ together with a group homomorphism $\rho: G \to \text{End}(V)$ such that $\rho(g_1) \rho(g_2) = \rho(g_1 \circ g_2)$ for any $g_1, g_2 \in G$. Given two representations $(V_1, \rho_1)$ and $(V_2, \rho_2)$, the tensor product representation is given by $\rho_1 \otimes \rho_2 : G \to \text{End}(V_1 \otimes V_2)$ such that $(\rho_1 \otimes \rho_2)(g_1 \otimes g_2) = \rho_1(g) \otimes \rho_2(g)$.

In general, the tensor product representations are \emph{reducible} because in the space $V_1 \otimes V_2$, there are invariant subsets for which acting by the tensor product representation only maps to other points in the subset. Representations without nontrivial invariant subspaces are \emph{irreducible}. The Clebsch-Gordan (CG) decomposition refers to the task of decomposing a tensor product into its irreducible representations. We examine the particular case of $SL(2,\mathbb{R})$.

The space of symmetric tensors in $F(2,d) \otimes F(2,d)$ is isomorphic as an $SL(2, \mathbb{R})$ module to the space of $(d+1) \times (d+1)$ symmetric matrices, $S^{d+1}$, on which the action is congruence by $A^{[d]}$,
\begin{equation*}
A^{[d]}X A^{[d]^T},    
\end{equation*}
for $A \in SL(2, \mathbb{R}), X \in S^{d+1}$. The correspondence is explicit if we identify polynomials with their coefficients (meaning $p = x^{[d]^T}c$ is identified with the vector $c$) and compute the tensor products as Kronecker products ($p \otimes p = \text{kron}(c, c)$, written also as $c \otimes c$). 
For example, in the unscaled monomial basis if
\begin{equation*}
A = \pmat{    5   & 28&   -26\\
28   &173  &-178\\
-26  &-178  & 200} = \pmat{1 \\ 2\\ 2}\pmat{1 \\ 2\\ 2}^T + \pmat{2 \\ 13\\ -14}\pmat{2 \\ 13\\ -14}^T,
\end{equation*}
$p = x^2 + 2xy + 2y^2$ and $q = 2x^2 + 13 xy - 14y^2$, then $A$ corresponds to $p \otimes p + q \otimes q$. Of course, there are other factorizations of $A$, and those give sums of tensors that in total are equal to the element $p\otimes p + q\otimes q$ in $F(2, d) \otimes F(2,d)$.

\subsubsection{Clebsch-Gordan Decomposition of $SL(2, \mathbb{R})$}\label{sec:cgdecomp}
We recall the Clebsch-Gordan (CG) decomposition of $SL(2, \mathbb{R})$ and give a new way of computing it. The finite-dimensional irreducible representations of $SL(2, \mathbb{R})$ are $F(2, d)$ for every nonnegative integer $d$. The Clebsch-Gordan decomposition is given by 
\begin{equation}\label{eqn:cgsl2}
    F(2, d_1) \otimes F(2, d_2) \simeq F(2, d_1 + d_2) \oplus F(2, d_1 + d_2 - 2) \cdots \oplus F(2,0)
\end{equation}
and realized explicitly by the classical \emph{transvectant}. Recall Cayley's $\Omega$ process
\begin{equation*}
    \Omega = \frac{\partial}{\partial x_1} \otimes \frac{\partial}{\partial x_2} - \frac{\partial}{\partial x_2} \otimes \frac{\partial}{\partial x_1},
\end{equation*} so that the $n$th transvectant map is given by
\begin{equation*}\label{eqn:transvectant}\psi_n(p,q) := Tr(\Omega^n (p \otimes q)), \end{equation*} where the trace operation means to treat the tensor product instead as standard multiplication (in effect, removing the tensor operation). 
When the tensor is symmetric, which is the case in the rest of this paper, the odd order transvectant will vanish due to its anti-symmetry. These formulas are mentioned in many places, e.g. \cite[\S 5]{olver1999classical}, \cite{bohning2010clebsch, lawrence2023learning} where the major differences relate to scaling factors and the way in which tensor products are represented. 

By identifying $S^{d+1}$ with $F(2,d) \otimes F(2,d)$ as explained in Section~\ref{sec:prelimtensorproducts}, it is possible to define a transvectant operation directly on $A \in S^{d+1}$ by applying the differentiation operators in $\Omega^2$ to the basis directly. The scaling matrices here are given for the \emph{unscaled} monomial basis. Let $D_1$, $D_2$, and $D_3$ be $(d-1) \times (d-1)$ diagonal matrices such that $(D_1)_{ii}=(d-i+1)(d-i)$, $(D_2)_{ii} = i (i+1)$, and $(D_3)_{ii} = (d-i)i$.
The notation $A_{R,C}$ means to delete the $r\in R$ rows and $c \in C$ columns from $A$. For an example of this operator being applied, see Example~\ref{example:thmcomputation}. \begin{defn} For $d \geq 2$, $\mathcal{T}: S^{d+1} \to S^{d-1}$ is given by
\begin{equation}
    d^2 (d-1)^2 \mathcal{T}(A) =  D_1 A_{\{d,d+1\},\{1,2\}} D_2 - 2D_3 A_{\{1,d+1\},\{1,d+1\}}D_3 + D_2 A_{\{1,2\},\{d,d+1\}} D_1.
\end{equation}
\end{defn}

Three equivalent viewpoints can express how a matrix may be supported only on certain irreps of $SL(2)$. The first two are based on the interpretation of \eqref{eqn:cgsl2} as the decomposition of a matrix $A$ in terms of binary forms
\begin{equation}\label{eqn:transvectant-matrix-decomp}
    A \leftrightarrow x^{[d]^T}Ax^{[d]} + x^{[d-2]^T}\mathcal{T}(A)x^{[d-2]} + x^{[d-4]^T}\mathcal{T}^2(A)x^{[d-4]} + \cdots + \mathcal{T}^{d/2}(A),
\end{equation}
which is interpreted to end at $\mathcal{T}^{d/2}$ if $d$ is even and $x^T \mathcal{T}^{(d-1)/2}(A)x$ if $d$ is odd. We say that a matrix is supported in some of the irreps of $SL(2)$ if only the corresponding forms in \eqref{eqn:transvectant-matrix-decomp} are nonzero.  
A similar viewpoint is that if $A$ is in the kernel of $\mathcal{T}^j$ for some $j$, then it cannot be supported in the ``rest'' of the irreps because applying $\mathcal{T}$ to a zero matrix will again give a zero matrix. 
A final formulation of this notion of support is that if $A = \sum_{i=1}^{d+1} s_i a_i a_i^T$ (for $s_i = \pm 1$) then perhaps only for some $k$ is $\sum_{i} s_i \psi_k(x^{[d]^T}a_i, x^{[d]^T}a_i)$ nonzero. We may interpret the statements in Section~\ref{sec:gram-matrix-structure} with any of these viewpoints. 

The only operation missing to interchange matrices and tensor products of polynomials is an operation between matrices that is equivalent to multiplying the corresponding polynomials. In Section~\ref{sec:tools} we prove existence of such an operation by explicitly constructing it.

\section{Building Gram matrices}\label{sec:tools} 
We discuss two tools that can make explicit our construction of Gram matrices with special properties. These tools are already useful in their own right, as demonstrated in \cite{johnston2023hierarchy}. We give slightly simpler and more explicit formulas for these constructions, but the central ideas are not new. 



\subsection{Canonical Gram matrix}

The first tool is a principled, explicit construction of a Gram matrix for a polynomial. In contrast to the proposal by \cite{lasserre2001global, parrilo2000structured} of using a semidefinite programming (SDP) solver to find a psd Gram matrix to certify a polynomial is a sum of squares, we could just make a (good) guess for such a matrix. One such guess is as follows.

The only linear equivariant maps from $F(2,d) \to S^{d+1} = F(2,d) \otimes F(2,d) \simeq F(2,2d) \oplus \cdots \oplus F(2, 0)$ are zero on all the components other than $F(2,d)$ by Schur's Lemma. The scale factor allowed by Schur's lemma should be 1 so that the matrix is actually a Gram matrix. While the existence of such a map was mentioned as early as \cite{cukierman2007positive}, we recount an explicit construction of the so-called \emph{Canonical} Gram matrix $G[p]$ as given by \cite{parrilo2021canonical}. Let $\gamma$ be a multi-index with $x^\gamma = \prod_{i=1}^n x_i^{\gamma_i}$ and let $p = c_\gamma x^\gamma \in F(n,2d)$, then $G[p]$ is the symmetric ${n + d \choose d} \times {n + d \choose d}$ matrix with rows and columns indexed by monomials so that
\begin{equation*}
    (G[p])_{\alpha, \beta} = c_{\alpha + \beta}\frac{\sqrt{\alpha! \beta!} }{(\alpha + \beta)!},
\end{equation*}
where $\alpha! = \frac{(\sum_i \alpha_i)!}{\prod_i (\alpha_i!)}$ is the multinomial coefficient with the standard factorials in the numerator and denominator. In the unscaled monomial basis case, the only change is to remove the square root in the numerator. The most notable property of $G$ is that by construction
\begin{equation}
    G[p(Ax)] = A^{[d]^T}G[p]A^{[d]}.
\end{equation}

While such a map has many potential applications, we use it to easily construct a psd Gram matrix for any form that is a sum of powers of linear forms. This construction provides a psd Gram matrix for such forms because $G$ is linear in its argument and it is rank one for a power of linear form. 
\begin{prop}\label{prop:Gpoweroflinearform}
    Let $c \in \mathbb{R}^n$. In the scaled monomial basis, $G[(c^T x)^{2d}]=c^{[d]} c^{[d]^T}$.
\end{prop}
The proof of this proposition is the explicit calculation provided by \cite{parrilo2021canonical}.
\begin{proof}
First, note that $(c^T x)^{2d} = \sum_{\gamma} \gamma! c^\gamma x^\gamma$. Then $G[(c^T x)^{2d}]_{\alpha, \beta} = (\alpha + \beta)! c^{\alpha + \beta} \frac{\sqrt{\alpha! \beta!}}{(\alpha + \beta)!} = \sqrt{\alpha! \beta!} c^{\alpha + \beta} = c^{[d]} c^{[d]^T}$.  
\end{proof}

The construction of $G$ allows us to state the simplest case of the kernel of $\mathcal{T}$, as mentioned at the end of Section~\ref{sec:cgdecomp}.
\begin{prop}
$\text{ker}(\mathcal{T}) = \{G[p] \, | \, p \in F(2,2d)\}$.
\end{prop}
\begin{proof}
Since $\text{dim}(\text{ker}(\mathcal{T})) = \text{dim}(S^{d+1}) - \text{rank}(\mathcal{T}) \leq \frac{(d+1)(d+2)}{2} - \frac{d(d-1)}{2} = 2d+1$, this provides an upper bound on the dimension of the subspace. On the other hand, the transvectants of powers of linear forms vanish, so these matrices are all in the kernel. Therefore, this is the whole kernel.
\end{proof}

\subsection{Symmetric tensors} 
We define the symmetric tensor product $\tot$ below for $p, q \in F(n,2d)$ so that it satisfies
\begin{equation}
x^{[d]^T}(M(p) \tot M(q))x^{[d]} = p(x) \cdot q(x)
\end{equation}
in the unscaled monomial basis and preserves certain spectral properties of $M(p)$ and $M(q)$, stated precisely in Lemma~\ref{lem:spectral-preserve}. Here we write $M(r)$ to mean \emph{any} Gram matrix for $r$. 
\begin{defn}
 Let $n$ be the ambient polynomial dimension. Let $a= \binom{d_1 + n - 1}{d_1}$ and $b = \binom{d_2 + n - 1}{d_2}$.  Suppose $A \in S^{a}$ and $B \in S^{b}$. Each row and column of $A$ ($B$) is indexed by a monomial in $n$ variables of degree $d_1$ ($d_2$). Define $A \tot B$ to be the symmetric matrix with side length $\binom{d_1 + d_2 + n- 1}{d_1 + d_2}$ whose entries are indexed by monomials in $n$ variables with degree $d_1 + d_2$ and given by
    \begin{equation}
        (A \tot B)_{(\alpha, \beta)} = \sum_{\substack{\alpha_1 + \alpha_2 = \alpha \\
        \beta_1 + \beta_2 = \beta}}A_{(\alpha_1,\beta_1)} B_{(\alpha_2, {\beta_2})}.
    \end{equation}
\end{defn}
This symmetric tensor product is a symmetrized version of the standard Kronecker product $A \otimes B$. Each row or column is indexed by a monomial, say $\gamma_i$. A row indexed by $\gamma$ in $A \tot B$ is the sum of rows in $A \otimes B$ such that $\gamma_i + \gamma_j = \gamma$. The same collection $\{(i,j)\}$ indexes the columns that should be summed to get column $\gamma$. This procedure gives us the matrix representing the product of the polynomials represented by each of the factors. 
\begin{prop}\label{prop:tot-multiply-polys}
    Let $p \in F(n, d_1)$, $q \in F(n,d_2)$, and $d = d_1 + d_2$. If $A = M(p)$ and $B = M(q)$ then $x^{[d]^T} (A \tot B) x^{[d]} = p(x) q(x)$. 
\end{prop}
\begin{proof}Consider the expansion
\begin{equation}
        \begin{aligned}
        x^{[d]^T} (A \tot B) x^{[d]} &= \sum_{\alpha, \beta}  x^{\alpha + \beta}\sum_{\substack{\alpha_1 + \alpha_2 = \alpha \\
        \beta_1 + \beta_2 = \beta}}A_{( {\alpha_1},  {\beta_1})} B_{( {\alpha_2},  {\beta_2})} \\
        &= \sum_{\alpha_1,\beta_1,\alpha_2,\beta_2}(x^{\alpha_1}A_{( {\alpha_1},  {\beta_1})}  x^{\beta_1})(x^{\alpha_2}B_{( {\alpha_2},  {\beta_2})} x^{\beta_2})\\
        &= \sum_{\alpha_1,\beta_1}(x^{\alpha_1}A_{( {\alpha_1},  {\beta_1})}  x^{\beta_1})\sum_{\alpha_2, \beta_2}(x^{\alpha_2}B_{( {\alpha_2},  {\beta_2})} x^{\beta_2})\\
        &= p(x) q(x).
    \end{aligned}
\end{equation}
\end{proof}
\begin{example}
    Let $A = \pmat{1 & 0 & 1\\ 0 & 2 & 0\\ 1 & 0 & 3}$ be a gram matrix of $p = x^4 + 4 x^2 y^2 + 3y^4$ and $B = \pmat{1 & 0\\ 0 & 0}$ be a gram matrix of $q = x^2$. (In this example we use the unscaled monomial basis). Then
    \begin{equation}
        A \tot B = \pmat{1 & 0 & 1 & 0 \\ 0 & 2 & 0 &0 \\ 1 & 0 & 3 & 0\\0 & 0 & 0 &0}
    \end{equation}
    is a gram matrix for $x^6 + 4x^4y^2 + 3x^2y^4 = p \cdot q$. 
\end{example}
The operation $\tot$ gives a systematic way to calculate a gram matrix of a polynomial given gram matrices of its factors. Versions of this operation have appeared before. For example, \cite{johnston2023hierarchy} denotes this operation by projection of the tensor product onto the symmetric part, and \cite{elsner2011bialternate} explicitly defines a similar product that differs from our definition by scaling in the case $d_1 = d_2 = 1$. 
This matrix operation enjoys several additional convenient properties.

\begin{prop}\label{prop:tot}
    Let $A \in S^a$, $B \in S^b$, $C \in S^c$, and $E \in S^e$. The following statements are true.
    \begin{enumerate}
        \item $A \tot B = B \tot A$.
        \item If $b = c$, then $A \tot (B + C) = A \tot B + A \tot C$.
    \end{enumerate}
\end{prop}
\begin{proof}
    The first property follows from commutativity of addition and multiplication. The second follows from distributivity.
\end{proof}

For our purposes, the most important property of $A \tot B$ is that it preserves the following spectral property of $A$ and $B$.
\begin{lem}\label{lem:spectral-preserve}
        If $A, B \succeq 0$ then $A \tot B \succeq 0$.
\end{lem}
\begin{proof}
    As discussed above, we can repeatedly apply the projection operator for each symmetrized row/column $\gamma$. Every one of these operations preserves positive semidefinitess. Since $A \otimes B \succeq 0$ and each partially symmetrized version of $A \otimes B$ preserves this property, we get that the final result $A \tot B$ is positive semidefinite as well. 
\end{proof}

\section{Gram matrix structure}\label{sec:gram-matrix-structure}

In this section, we can combine the matrix constructions to prove a correspondence between irreducible representations of $SO(2)$ and $SL(2)$. Recall the description of a symmetric matrix having support on only some components of the irreps of $SL(2)$ from Section~\ref{sec:cgdecomp}. 
\begin{thm}\label{thm:main}
    Let $p = q_1 \cdot q_2$ be a sum of squares. If $q_1$ has degree $2d_1$ and $q_2$ has degree $2d_2$, and $q_1$ is a sum of powers of linear forms, then there exists a psd Gram matrix for $p$ that is supported only on the $\lceil d_2/2 \rceil +1$ components \begin{equation*}F(2, 2(d_1 + d_2)) \oplus F(2, 2(d_1 + d_2) - 4) \oplus \cdots \oplus F(2, 2d_1 - (2d_2 \pmod 4)).
    \end{equation*}
\end{thm}

The main result follows from this property because $\Vert x \Vert^{2d}$ is a sum of powers of linear forms (e.g. \cite{reznick1992sum}), and therefore we also have $G[\Vert x \Vert^{2d}] \succeq 0$ by Proposition \ref{prop:Gpoweroflinearform} or completely independently by Proposition 4 of \cite{johnston2023hierarchy}.  
\begin{cor}\label{cor:main}
    Suppose that $p \in H_{2,0}\Vert x \Vert^{2d} \oplus H_{2,2} \Vert x \Vert^{2d-2} \oplus \cdots \oplus H_{2,2k}\Vert x \Vert^{2d-2k} \subset F_{2,2d}$ is a sum of squares. Then there exists a psd Gram matrix of $p$ whose decomposition is suported on the $\lceil k / 2 \rceil + 1$ components \begin{equation*}F(2, 2d) \oplus F(2, 2d-4) \oplus  \cdots \oplus F(2, 2d-2k - (2k \pmod{4}).\end{equation*}
\end{cor}

Next we present a series of technical lemmas that are required to prove this theorem. In the following lemmas, we write $f = 0$ to mean $f$ is \emph{identically zero} so $\frac{d^k}{dx^k} f(x) = 0$ for all $k$ (not just equal to zero for some setting of parameters).
    \begin{lem}\label{lem:transvectant_linear_form_power0}
        Let be $p = (ax_1 + bx_2)^{d}$ be the power of a linear form. If $q,r \in F(2,k)$ then $\Omega^{2(k+1)}(pq \otimes pr) = 0$. 
    \end{lem}
    \begin{proof}
        We argue that each application of the $\Omega$ operator reduces the degree of at least one of the factors $q$ or $r$. The degrees of $q$ and $r$ are at most $k$. Then
    \begin{equation}
            \begin{aligned}
            \Omega(qp \otimes r p) &= (\frac{\partial}{\partial x_1} \otimes \frac{\partial}{\partial x_2} - \frac{\partial}{\partial x_2} \otimes \frac{\partial}{\partial x_1}) (qp \otimes r p)\\ &= ((\frac{\partial q}{\partial x_1}p + q\frac{\partial p}{\partial x_1}) \otimes (\frac{\partial r}{\partial x_2}p + r\frac{\partial p}{\partial x_2}) - (\frac{\partial q}{\partial x_2}p + q\frac{\partial p}{\partial x_2}) \otimes (\frac{\partial r}{\partial x_1}p + r\frac{\partial p}{\partial x_1}))\\
            &= ((\frac{\partial q}{\partial x_1}p + q\frac{\partial p}{\partial x_1}) \otimes (\frac{\partial r}{\partial x_2}p + r\frac{\partial p}{\partial x_2}) - (\frac{\partial q}{\partial x_2}p + q\frac{\partial p}{\partial x_2}) \otimes (\frac{\partial r}{\partial x_1}p + r\frac{\partial p}{\partial x_1})).
            \end{aligned}
        \end{equation}
        The only terms with no derivative of $q$ or $r$ are
        \begin{equation}
        \begin{aligned}
            &q \frac{\partial p}{\partial x_1} \otimes r \frac{\partial p}{\partial x_2} - q \frac{\partial p}{\partial x_2} \otimes r \frac{\partial p}{\partial x_1} \\ &= qda(ax_1 + bx_2)^{d-1} \otimes rbd(ax_1 + bx_2)^{d-1} - qdb(ax_1 + bx_2)^{d-1} \otimes rad(ax_1 + bx_2)^{d-1} \\&= 0.
            \end{aligned}
        \end{equation}
        Each of the other terms is also of the form $pq \otimes pr$ for a new $q,r$ and $d$. Since $\Omega$ is linear, we can apply term by term to see that at least one of the new $q$'s and $r$'s are of a lower degree. Therefore, after $2k$ applications, either one of $q$ or $r$ has been differentiated more than $k$ times (so that the term is zero) or have both been differentiated exactly $k$ times. In that case, we apply once more $\Omega^2(p \otimes p) = 0$ because the original $p$ and all its derivatives are the power of a linear form.

    \end{proof}
 
    \begin{lem}\label{lem:gptotm}
        Let $p \in F(n,2d_1)$ and $M \in S^{d_2 + 1}$. Then 
        \begin{equation}
            \mathcal{T}^{d_2+1}(G[p] \tot M) = 0.
        \end{equation}
    \end{lem}
    \begin{proof}
        First write $p = \sum s_i (p_i^T x)^{2d_1}$, as the decomposition of $p$ into powers of linear forms with $s_i = \pm 1$. Therefore $G[p]=\sum s_i p_i^{[d_1]} p_i^{[d_1]^T}$ (or simply congruent in the unscaled case). 
        Let us factor $M = \sum c_i m_i m_i^T$. Then 
        \begin{equation}
        \begin{aligned}
        G[p] \tot M &= \sum_{i,j} s_ic_j (p_i^{[d_1]} \otimes p_i^{[d_1]}) \tot (m_j \otimes m_j) 
        \end{aligned}
        \end{equation}
        and because $\mathcal{T}$ is linear we only must prove that $\mathcal{T}^{d_2 + 1}((p_i^{[d_1]} \otimes p_i^{[d_1]}) \tot (m_j \otimes m_j)) = 0.$ The $\alpha, \beta$ entry of $((p_i^{[d_1]} \otimes p_i^{[d_1]}) \tot (m_j \otimes m_j))$ is given by
        \begin{equation}
            \left(\sum_{\alpha_1 + \alpha_2 = \alpha} (p_i^{[d_1]})_{\alpha_1}(m_{j})_{\alpha_2}\right) \left( \sum_{\beta_1 + \beta_2 = \beta} (p_i^{[d_1]})_{\beta_1}(m_{j})_{\beta_2}\right).
    \end{equation} So the whole matrix is the outer product of the two vectors indexed by $\alpha$ and $\beta$ respectively whose entries are the convolutions of the entries of $p_i^{[d_1]}$ and $m_j$, which can be calculated via polynomial multiplication. 
        In the language of polynomials, we therefore compute $\Omega^{2(d_2 + 1)}$ applied to  \begin{equation}(x^{[d_1]^T}p_i^{[d_1]} \cdot x^{[d_2]^T}m_j) \otimes (x^{[d_1]^T}p_i^{[d_1]} \cdot x^{[d_2]^T}m_j)  = ((p_i^Tx)^{d_1} \cdot x^{[d_2]^T}m_j) \otimes ((p_i^Tx)^{d_1} \cdot x^{[d_2]^T}m_j) , \end{equation}which vanishes by Lemma~\ref{lem:transvectant_linear_form_power0}.
    \end{proof}
The main theorem of this section follows directly from the preceding Lemma.
\begin{proof}[Proof of Theorem~\ref{thm:main}] Since $p$ is positive, $q_1$ must be as well. As a nonnegative binary form, $q_1$ is a sum of squares and so has a psd Gram matrix, which we call $M$. Since $q_2$ is a sum of powers of linear forms, $G[q_2]$ is psd by Proposition~\ref{prop:Gpoweroflinearform}. We claim the matrix
\begin{equation}\label{eqn:GtotM}
    G[q_2] \tot M
\end{equation}
has the properties we claimed. By Proposition~\ref{prop:tot-multiply-polys}, this is a valid Gram matrix. By Lemma~\ref{lem:spectral-preserve} it is psd. Finally, by Lemma~\ref{lem:gptotm} it has the claimed support.
\end{proof}

We calculate a simple example of Corollary~\ref{cor:main} to demonstrate the construction used in the proof of Theorem~\ref{thm:main}. 
\begin{example}\label{example:thmcomputation}
    Let $p = (x_1^2 + x_2^2)^3(2x_1^2-2x_1x_2+5x_2^2)$. The Harmonic decomposition, $3.5(x_1^2 + x_2^2)^4 + (1.5x_2^2 - 2x_1x_2 - 1.5x_1^2)(x_1^2 + x_2^2)^3$, is supported on just two components: $H_{2,0}\Vert x \Vert^{8} \oplus H_{2, 2}\Vert x \Vert^{6}$. The psd Gram matrix in \eqref{eqn:GtotM} in the unscaled monomial basis is
    \begin{equation}
        \pmat{1 & 0 & 3/5 & 0\\ 0 & 9/5 & 0 & 3/5 \\ 3/5 & 0 & 9/5 & 0 \\ 0 & 3/5 & 0 & 1} \tot \pmat{2 & -1 \\ -1 & 5}= \pmat{2 & -1 & 6/5 & -3/5 & 0 \\ -1 & 43/5 & -12/5 & 21/5 & -3/5 \\ 6/5 & -12/5 & 63/5 & -12/5 & 3 \\ -3/5 & 21/5 & -12/5 & 11 & -1\\ 0 & -3/5 & 3 & -1 & 5}.
    \end{equation}
    Applying $\mathcal{T}$ with $d = 4$ directly we get $\frac{1}{144}$ times
    \begin{equation}
    \begin{aligned}
       &\pmat{12 & 0 & 0\\ 0 & 6 & 0 \\ 0 & 0 & 2} \pmat{6/5 & -3/5 & 0 \\ -12/5 & 21/5 & -3/5 \\ 63/5 & -12/5 & 3} \pmat{2 & 0 & 0\\ 0 & 6 & 0\\ 0 & 0 & 12} \\
       -&2 \pmat{3 & 0 & 0\\ 0 & 4 & 0 \\ 0 & 0 & 3} \pmat{43/5 & -12/5 & 21/5 \\ -12/5 & 63/5 & -12/5\\ 21/5 & -12/5 & 11} \pmat{3 & 0 & 0\\ 0 & 4 & 0\\ 0 & 0 & 3} \\+ 
       &\pmat{2 & 0 & 0\\ 0 & 6 & 0\\ 0 & 0 & 12} \pmat{6/5 & -12/5 & 63/5\\ -3/5 & 21/5 & -12/5 \\ 0 & -3/5 & 3} \pmat{12 & 0 & 0\\ 0 & 6 & 0 \\ 0 & 0 & 2},  
       \end{aligned}
    \end{equation}
    which is 
    \begin{equation}
        \frac{-1}{40}\pmat{27 & 4 & 7 \\ 4 & 28 & 4 \\ 7 & 4 & 15}.
    \end{equation}
    Applying $\mathcal{T}$ once more with $d = 2$, we get 
    \begin{equation}
        \frac{-1}{40} \cdot \frac{1}{2}(2 \cdot 7 \cdot 2 - 2 ( 1 \cdot 28 \cdot 1) + 2 \cdot 7 \cdot 2) = 0,
    \end{equation}
    and so the Gram matrix is supported on just $\lceil \frac{2}{2}\rceil + 1 = 2$ components as the corollary suggests.
\end{example}
\section{Conclusion}

In this work, we have identified a structural relationship in some positive polynomials that may carry over into the structures of their Gram matrices. While this relationship is currently only demonstrated for the case of binary forms, an analog for polynomials in higher dimensions that are acted on by $SL(n)$ and $SO(n)$ is the subject of future work.

\section*{Acknowledgement} The author would like to thank Pablo Parrilo for many helpful discussions.

\bibliography{references}
\bibliographystyle{plain}
\appendix

\end{document}